\documentclass[8pt]{article}

\usepackage{amsmath}

\usepackage{amssymb}

\usepackage{epsfig}
\numberwithin{equation}{section}
\newtheorem{theorem}{Theorem}[section]
\newtheorem{lemma}{Lemma}[section]

\newtheorem{remark}{Remark}[section]
\newtheorem{assumption}{Assumption}[section]

\def\QEDopen{{\setlength{\fboxsep}{0pt}\setlength{\fboxrule}{0.2pt}\fbox{\rule[0pt]{0pt}{1.3ex}\rule[0pt]{1.3ex}{0pt}}}} 

\def\QED{\QEDopen} 

\def\endproof{\hspace*{\fill}~\QED\par\endtrivlist\unskip}

\begin{document}
\title{Jump stochastic differential equations with non-Lipschitz and superlinearly growing coefficients\footnotemark[1]}
\author{Yuchao Dong\footnotemark[2]}
\maketitle
\renewcommand{\thefootnote}{\fnsymbol{footnote}}

\footnotetext[1]{This research was supported by the National Natural Science Foundation of China (Grants \#11171076), and by Science and Technology Commission, Shanghai Municipality (Grant No.14XD1400400). The author would like to thank his advisor, Prof. Shanjian Tang from Fudan University, for the helpful comments and discussions.}

\footnotetext[2]{Department of Finance and Control Sciences, School of Mathematical Sciences, Fudan University, Shanghai 200433, China. yuchaodong13@fudan.edu.cn (Yuchao Dong).}
\maketitle
\abstract{In the paper, we consider the no-explosion condition and pathwise uniqueness for SDEs driven by a Poisson random measure with coefficients that are super-linear and non-Lipschitz. We give a comparison theorem in the one-dimensional case under some additional condition. Moreover, we  study the non-contact property and the continuity with respect to the space variable of the stochastic flow. As an application, we will show that there exists a unique strong solution for SDEs with coefficients like $x\log|x|$.}
\section{Introduction and preliminaries}
Consider the following jump-type stochastic differential equation (SDE):
\begin{equation}
X_t=x+\int_0^tf(X_{s-})ds+\int_0^tg(X_{s-})dW_s+\int_0^t\int_{\mathcal U}h(X_{s-},u)\tilde N(du,ds).\label{SDE}
\end{equation}
 Here $x \in \mathbb R^m$, $W_t$ is a Brownian motion on $\mathbb R^n$ and $\tilde N(du,dt)$ is a compensated Poisson random measure. The functions $f$, $g$, $h$ are defined as
\begin{equation*}
\begin{aligned}
f:\mathbb R^m \to \mathbb R^m,
g:\mathbb R^m \to \mathbb R^{m \times n}, \text{ and }
h:\mathbb R^m \times \mathcal U \to \mathbb R^m.
\end{aligned}
\end{equation*}
The purpose of this paper is to  study the no-explosion condition and pathwise uniqueness for the solutions of (\ref{SDE}) with super-linear and non-Lipschitz coefficients  including the functions like $x\log|x|$. Some other properties of the solutions, including comparison principle, non-contact property and continuity with respect to $x$ of the stochastic flow, are also proved under some additional assumptions. As an application, we prove that the following SDE:
\begin{equation*}
\begin{aligned}
X_t=x&+\int_0^{t}X_{s-}\log|X_{s-}|ds +\int_0^t X_{s-}\sqrt{|\log|X_{s-}||}dW_s\\
&+\int_0^{t}\int_{\mathcal U}X_{s-}\sqrt{|\log|X_{s-}||}u\tilde N(du,ds),
\end{aligned}
\end{equation*}
possesses a unique strong solution. \\
\quad\\
SDEs with jumps plays an important role in many applications. For example, in a  financial market,  a stock price is described  by a process that admits random jumps to characterize a sudden shift of the price. Under linear growth and Lipschitz conditions, the existence and the uniqueness of strong solutions for jump-type SDEs can be established by arguments based on fixed-point theory and Gronwall lemma: see e.g. \cite{Ikeda_Watanabe_1989}. There are numerical papers discussing the existence and uniqueness of weak or strong solutions in  relaxed conditions. For the following equation
\begin{equation}\label{SDE.example}
dX_t=F(X_{t-})dL_t,
\end{equation}
Bass \cite{Bass_2003} and Komatsu \cite{Komatsu_1982} showed that (\ref{SDE.example}) admits a unique strong solution if $L_t$ is a symmetric $\alpha$-stable process with $1 < \alpha <2$ and the coefficient $F(\cdot)$ is continuous and bounded with the modulus of continuity $z \to \rho(z)$ satisfying
\[\int_{0+}\frac{dz}{\rho(z)^{\alpha}}=+\infty.\]
This result is an extension  of the Yamada-Watanabe criterion for SDEs driven by Brownian motions, see \cite{Karatzas_Shreve_1991}. Fu and Li \cite{Fu_Li_2010} studied the existence and uniqueness of strong solutions for (\ref{SDE.example}) driven by spectrally positive $\alpha$-stable L\'evy noises which included the case $F(x)=x^{\frac{1}{\alpha}}$. It was extended by Li and Mytnik \cite{Li_Mytnik_2011} to more general cases. All these three works require that the coefficients have a linear growth, and thus the coefficients like $x\log|x|$ are not included. The results of Yamada-Watanabe (weak existence and pathwise uniqueness imply uniqueness in the sense of probability law and strong existence for the SDEs) has also been extended to jump-type SDEs. See  \cite[Theorem 137]{Situ_2005} and Barczy et al \cite{Barczy_Li_Pap_2015}. For more details as well as  applications
to finance, see the survey paper of Bass \cite{Bass_2004} and the monograph of Applebaum \cite{Applebaum_2004}.\\
\qquad\\
Compared with \cite{Fu_Li_2010} and \cite{Li_Mytnik_2011}, our SDE is allowed to have a super-linear growth with respect to the unknown variable. Moreover, our pathwise uniqueness result for the solutions do not require that the function $h$ is non-decreasing with respect to $x$ which is essential in \cite{Fu_Li_2010} and \cite{Li_Mytnik_2011}.However, the non-decreasing condition is still required in our comparison principle.  The stochastic flow associated with (\ref{SDE}) are also considered. For the case $h=0$ in SDE (\ref{SDE}), Fang and Zhang \cite{Fang_Zhang_2005}  proved the non-contact property (i.e. if $x \neq y$, then almost surely $X_t(x)\neq X_t(y)$ for all $t>0$) and the comparison principle in the one-dimensional case (i.e.  if $x \le y$, then $X_t(x) \le X_t(y)$). Furthermore, it holds that
\begin{equation}\label{SDE.introduction}
\lim_{|x| \to +\infty}|X_t(x)|=+\infty, \text{ in probability.}
\end{equation}
Since the appearance of random jumps, these properties do not hold in general for our SDEs. Here are two counterexample. Consider the following SDE:
\begin{equation}\label{SDE.jump}
X_t=x+\int_0^t\int_{\mathcal U} (-X_{s-})N(du,ds)
\end{equation}
where  $N(du,dt)$ is a Poisson random measure with $\mu(\mathcal U) <+\infty$.\\
Define $\tau_n:=\inf\{t>0;N(\mathcal U \times (0,t])=n\}$ and $\tau_0=0$. The solution for (\ref{SDE.jump}) is  $X_t=x1_{\tau_1 > t}$. It is easy to see that (\ref{SDE.introduction}) and the non-contact property do not hold, since $X_{\tau_1}=0$ for all $x$.  Consider another example:
\begin{equation}\label{SDE.comparison}
X_t=x+\int_0^t \int_{\mathcal U}(-2X_{s-})N(du,ds).
\end{equation}
If $x=1$, then the solution $X$ is $X_t=\sum_{n=0}^{+\infty}(-1)^n1_{\tau_n \le t <\tau_{n+1}}$. For $x=-1$, the solution $\tilde X$ is $\tilde X_t=\sum_{n=0}^{+\infty}(-1)^{(n+1)}1_{\tau_n \le t <\tau_{n+1}}$. Thus the comparison principle is violated for these two solutions. Note that Bahlali et al \cite{Bahlali_Eddahbi_Essaky_2003} proved a comparison principle  for backward stochastic differential equations (BSDEs) with random jumps.  To summarize, as in \cite{Kunita_2004}, we prove the comparison principle for a non-decreasing coefficient $h$ and the non-contact property under the assumption that  the mapping $x \to x+h(x,u)$ is homeomorphic.\\
\qquad\\
We end up this section with introducing some notations. Let $(\Omega,\mathcal F, \{\mathcal F_t\},\mathcal P)$ be a filtered probability space satisfying the usual hypotheses. $W_t$ is a standard $\{\mathcal F_t\}$-Brownian motion on $\mathbb R^m$. Let $\mathcal U$ be a metric space and $\mu$ a $\sigma$-finite measure on $\mathcal U$. $N(du,dt)$ is a $\{\mathcal F_t\}$-adapted Poisson random measure with intensity measure $\mu(du)dt$. $\tilde N(du,dt):=N(du,dt)-\mu(du)dt$ is the associated compensated Poisson random measure. Suppose that $\{W_t\}$ and $N(du,dt)$ are mutually independent. A strong solution of (\ref{SDE}) with its lifetime $\tau$ is an  $\{\mathcal F_t\}$-adapted c\`adl\`ag process $\{X_t\}$ and a stopping time $\tau$ such that: \\
\qquad(1) the stopping time $\tau$ is defined as $\tau:=\lim_{n \to +\infty}\tau_n$ with $\tau_n:=\inf\{t>0,|X_t|\ge n\}$,\\
\qquad(2) the equation (\ref{SDE})  holds for all $t < \tau$ almost surely.\\
Note that one can use $X_s$ instead of $X_{s-}$ in the integral with respect to $dt$ and $dW_t$ on the right hand side of (\ref{SDE}), since $X_s \neq X_{s-}$ for at most countable $s$.\\
\quad\\
For a matrix $A$, we denote by $\left \|A \right \|$ the Hilbert-Schmidt norm: $\left \|A \right \|^2=\text{tr}(A^TA)=\sum_{i,j}A^2_{ij}$; for a vector $x \in \mathbb R^m$, $|x|$ is the Euclidean norm. Given a function $f$ on $\mathbb R^m$ and $x,y \in \mathbb R^m$, we note
\[\Delta_yf(x):=f(y)-f(x) \text{ and } D_yf(x):=\Delta_yf(x)-\left \langle \nabla f(x),y-x\right \rangle.\]
Throughout the paper, the constant $C$ in an inequality is universal  and is allowed to change from line to line. Whenever necessary, we use $C(p)$ to emphasize its dependence on the parameter $p$.
\section{No-explosion condition for the SDEs}
In this section, we will show that the solution for (\ref{SDE}) do not explode under the following assumptions.
\begin{assumption}\label{assumption.nonexplosion}
There exists a nondecreasing $C^1$ function $\rho$ defined on a neighborhood $[K, +\infty]$ of $+\infty$ satisfying: \\
(1)$\lim_{s \to \infty}\rho(s)=+\infty$, (2) $\rho(s) \ge 1$, (3) $\lim_{s \to \infty}\frac{s\rho'(s)}{\rho(s)}=0$, \\(4) $\int_{0}^{\infty}\frac{1}{s\rho(s)+1}ds=+\infty$. \\
For all $x \in \mathbb R^m$, we have
\begin{equation*}
\begin{aligned}
&|f(x)| \le C(|x|\rho(|x|^2)+1),\\
&\left\| g(x) \right\|^2 \le C(|x|^2\rho(|x|^2)+1),\\
&\int _{\mathcal U}|h(x,u)|^2\mu(du)\le C(|x|^2\rho(|x|^2)+1).
\end{aligned}
\end{equation*}
\end{assumption}
By the property (3) in the assumption, for any $\varepsilon>0$, we see that $\rho'(s)\le \varepsilon \frac{\rho(s)}{s}$ for $s$ sufficiently large which implies that the coefficients can not increase faster than $|x|^{1+\varepsilon}$. This is also necessary in some sense for no-explosion condition. To see this, consider the function $x_t=(1-t)^{-\frac{1}{\varepsilon}}$ which satisfies the ODE:
\[dx_t=\frac{1}{\varepsilon}x_t^{1+\varepsilon},\]
and explodes at time $t=1$. Typical examples satisfying the above conditions are functions $\rho(s)=\log s$, $\rho(s)=(\log s \cdot \log \log s )$.
\begin{theorem}\label{thm.nonexplosion}
Let Assumption \ref{assumption.nonexplosion} be satisfied.  Then the solution for the SDE (\ref{SDE}) has no explosion, i.e. $\mathcal P(\tau=+\infty)=1$.
\end{theorem}
\begin{proof}
Without loss of generality, we assume that $K=0$ in Assumption \ref{assumption.nonexplosion}. Define the functions
\[\Psi(\xi)=\int_{0}^{\xi}\frac{1}{s\rho(s)+1}ds\text{ and }\Phi(\xi)=\exp(\Psi(\xi)), \qquad \xi\ge0.\]
We have
\begin{equation}\label{Phi.derivative}
\Phi'(\xi)=\frac{\Phi(\xi)}{\xi \rho(\xi)+1}\text{ and }\Phi''(\xi)=\frac{\Phi(\xi)(1-\rho(\xi)-\xi\rho'(\xi))}{\xi \rho(\xi)+1}.
\end{equation}
According to the assumptions, we see that $1-\rho(\xi)-\xi\rho'(\xi) \le 0$ which implies that $\Phi$ is a concave function. Thus one gets that
\begin{equation}\label{Phi.concave}
\Phi(y)-\Phi(x)-\Phi'(x)(y-x) \le 0,
\end{equation}
for any $x,y \ge 0$.\\
Let $X$ be a solution to (\ref{SDE}) and $\tau$ its lifetime. Define the stopping time
\[\tau_{R}:=\inf\{t>0,|X_t| \ge R \}, \text{ for any } R\ge0.\]
It is clear that $\tau_R$ tends to $\tau$ as $R \to \infty$. By It\^o-formula, we have
\begin{equation*}
\begin{aligned}
d|X_t|^2&=2\left\langle X_{t-},dX_t\right \rangle+\left \|g(X_{t-})\right\|^2dt+\int_{\mathcal U}|h(X_{t-},u)|^2N(du,dt)\\
&=\big\{2\left \langle X_{t-},f(X_{t-})\right \rangle+\left \|g(X_{t-})\right\|^2+\int_{\mathcal U}|h(X_{t-},u)|^2\mu(du)\big\}dt\\
&+2\left \langle X_{t-},g(X_{t-})\right \rangle dW_t+\int_{\mathcal U}[|X_{t-}+h(X_{t-},u)|^2-|X_{t-}|^2]\tilde N(du,dt).
\end{aligned}
\end{equation*}
and
\begin{equation*}
\begin{aligned}
d\Phi(|X_t|^2)&=\Phi'(|X_{t-}|^2)d|X_t|^2+2\Phi''(|X_{t-}|^2)|\left\langle X_{t-},g(X_{t-})\right \rangle|^2dt\\
&+\int_{\mathcal U}D_{|X_{t-}+h(X_{t-},u)|^2}\Phi(|X_{t-}|^2)N(du,dt).\\
\end{aligned}
\end{equation*}
Thus
\[\Phi(|X_t|^2)=I_1(t)+I_2(t)+M_t,\]
with
\begin{equation*}
\begin{aligned}
&I_1(t)=\int_0^t\Phi'(|X_{s-}|^2)\big\{2\left \langle X_{s-},f(X_{s-})\right \rangle+\left \|g(X_{s-})\right \|^2+\int_{\mathcal U}|h(X_{s-},u)|^2\mu(du)\big\}ds\\
&\qquad +\int_0^t2\Phi''(|X_{s-}|^2)|\left\langle X_{s-},g(X_{s-})\right \rangle|^2ds,\\
&I_2(t)=\int_0^t\int_{\mathcal U}D_{|X_{s-}+h(X_{s-},u)|^2}\Phi(|X_{s-}|^2)\mu(du)ds,\\
&M_t=\int_0^t2\Phi'(|X_{s-}|^2)\left \langle X_{s-},g(X_{s-})\right \rangle dW_s+\int_0^t\int_{\mathcal U}\Delta_{|X_{s-}+h(X_{s-},u)|^2}\Phi(|X_{s-})|^2)\tilde N(du,ds).
\end{aligned}
\end{equation*}
By the assumption, we have
\[2\left \langle X_{s-},f(X_{s-})\right \rangle+\left \|g(X_{s-})\right \|^2+\int_{\mathcal U}|h(X_{s-},u)|^2\mu(du) \le C(| X_{s-}|^2\rho(| X_{s-}|^2)+1).\]
According to (\ref{Phi.derivative}) and $\Phi''(\xi) \le 0$, we get that
\[I_1 \le C\int_0^t\Phi(|X_{s-}|^2)ds.\]
Moreover, (\ref{Phi.concave})  implies that $I_2(t) \le 0$.\\
Since the process is uniformly bounded before $\tau_R$, we see that the stochastic integral part $M_t$ is a martingale. Hence, after taking expectation, we have
\begin{equation}\label{estimation.Gronwall}
E[\Phi(|X_{t\wedge \tau_{R}}|^2)] \le \Phi(|X_0|^2)+C_1E[\int_0^t\Phi(|X_{(s\wedge \tau_{R})-}|^2)ds],
\end{equation}
for some constant $C_1$.\\
We can use $|X_{s\wedge \tau_{R}}|^2$ instead of $|X_{(s\wedge \tau_{R})-}|^2$ since $X_s \neq X_{s-}$ only for at most countable $s$. Thus
\[E[\Phi(|X_{t\wedge \tau_{R}}|^2)] \le \Phi(|X_0|^2)+C_1E[\int_0^t\Phi(|X_{s\wedge \tau_{R}}|^2)ds].
\]
By Gronwall lemma, we get that there exists a constant $C_2 >0$ such that for all $t \ge 0$ and $R >0$,
\begin{equation*}
E[\Phi(|X_{t\wedge \tau_{R}}|^2)] \le \Phi(|X_0|^2)e^{C_2T}.
\end{equation*}
Letting $R \to \infty$ in above inequality, by Fatou lemma, we get
\begin{equation}
 E[\Phi(|X_{t\wedge \tau}|^2)] \le \Phi(|X_0|^2)e^{C_2T}.\label{estimation.Phi}
 \end{equation}
Now if $P(\tau <\infty)>0$, one can find  a large $T>0$ such that $P(\tau \le T)>0$.
Taking $t=T$ in (\ref{estimation.Phi}), we get
\begin{equation}
E[1_{(\tau \le T)}\Phi(|X_\tau|^2)]\le \Phi(|X_0|^2)e^{C_2T}.\label{estimation.Phi2}
\end{equation}
Since $\Phi(|X_\tau|^2)=\Phi(\infty)=+\infty$ on a set with positive measure, the left hand side of (\ref{estimation.Phi2}) should be infinity which is a contradiction. Therefore we have
\[\mathcal P(\tau <\infty)=0.\]
\end{proof}
\begin{remark}
In addition to the assumptions of Theorem \ref{thm.nonexplosion}, if we further assume that
\[2\left \langle x,f(x)\right \rangle+\left \|g(x)\right \|^2+\int_{\mathcal U}|h(x,u)|^2\mu(du) \le0,\]
for all $x \in \mathbb R^m$. Then the analysis in the proof of Theorem \ref{thm.nonexplosion} indicates that
\[E[\Phi(|X_t|^2)] \le \Phi(|x|^2), \text{ for all $t>0$.}\]
Since $\lim_{|\xi| \to +\infty}\Phi(\xi)=+\infty$, it implies that the family of probability measures $P(t,x,dy):=\mathcal P(X_t(x) \in dy)$ is tight. By the theory of Krylov and Bogolubov (see Theorem 4.17 in \cite{Hairer_2006}), if the SDE (\ref{SDE}) also defines a Feller semigroup, then there exists an invariant measure for this Markov process.
\end{remark}
Let $X_t(x)$ be the solution of (\ref{SDE}) with initial value $x$.  By Theorem 4.1 in \cite{Fang_Zhang_2005}, we know that when there is no random jumps, i.e. $h=0$, one will have that
\begin{equation}
\lim_{|x|\to \infty}|X_t(x)|=\infty \text{ in probability.}\label{nocomeback}
\end{equation}
That is if the initial point of the process is far from the origin, then, for any compact set, the process will not enter into it in a short time. But this is not true when the SDE is also driven by a Poisson random measure. An counter example has been given in the previous section. To ensure that  (\ref{nocomeback}) holds, some additional assumptions must be added.
\begin{assumption}\label{assumption.nocomeback}
There exist $C>0$ such that
\[|h(x,u)|^2 \le C,\]
and $\mu(\mathcal U) < \infty$.
\end{assumption}
The meaning of the assumption is clear. It requires that the size of the jump is uniformly bounded and there will not be too many jumps for each path. We have the following result:
\begin{theorem}
In addition to the assumptions of Theorem \ref{thm.nonexplosion}, we also require that Assumption \ref{assumption.nocomeback} is satisfied. Then (\ref{nocomeback}) holds.
\end{theorem}
\begin{proof}
Let $\psi$ be defined as in the proof of Theorem \ref{thm.nonexplosion}. Define the function $\Phi$ by
\[\Phi(\xi)=\exp(-\psi(\xi)).\]
In this case, one gets that $\Phi'(\xi)=-\frac{\Phi(\xi)}{\xi\rho(\xi)+1}$ and for some $C_1>0$,
\begin{equation}\label{Phi.2nd.derivative}
\Phi''(\xi)=\frac{\Phi(\xi)(1+\rho(\xi)+\xi\rho'(\xi))}{(\xi\rho(\xi)+1)^2}\le C_1\frac{\Phi(\xi)\rho(\xi)}{(\xi\rho(\xi)+1)^2}.
\end{equation}
Let $R$ and $M$ be two constants such that $M > |x_0|>R$. Define
\[\tilde \tau_R:=\inf \{t>0,|X_t(x_0)| \le R\},\]
and
\[ \tau_M:=\inf \{t>0,|X_t(x_0)| \ge M\}.\]
As in the proof of Theorem \ref{thm.nonexplosion}, we have
\[E[\Phi(|X_{t\wedge\tilde \tau_R\wedge\tau_M}|^2)]=\Phi(|x_0|^2)+I_1(t)+I_2(t),\]
where
\begin{equation*}
\begin{aligned}
I_1(t)=E[\int_0^{t\wedge\tilde \tau_R\wedge\tau_M}&\bigg\{\Phi'(|X_{s-}|^2)\{2\left \langle X_{s-},f(X_{s-})\right \rangle+\left \|g(X_{s-})\right \|^2\\
&+\int_{\mathcal U}|h(X_{s-},u)|^2\mu(du)\}+2\Phi''(|X_{s-}|^2)|\left\langle X_{s-},g(X_{s-})\right \rangle|^2\bigg \}ds],
\end{aligned}
\end{equation*}
and
\[I_2(t)=E[\int_0^{t\wedge\tilde \tau_R\wedge\tau_M}\Phi(|X_{s-}+h(X_{s-},u)|^2)-\Phi(|X_{s-}|^2)-\Phi'(|X_{s-}|^2)(2\left\langle X_{s-},h\right\rangle+|h|^2)\mu(du)ds].\]
By Assumption \ref{assumption.nonexplosion} and (\ref{Phi.2nd.derivative}), we see that
\begin{equation*}
\begin{aligned}
\Phi''(|X_{t-}|^2)|\left\langle X_{t-},g(X_{t-})\right \rangle|^2 &\le
C\frac{\Phi(|X_{t-}|^2)\rho(|X_{t-}|^2)}{(|X_{t-}|^2\rho(|X_{t-}|^2)+1)^2}|X_{t-}|^2(|X_{t-}|^2\rho(|X_{t-}|^2)+1)\\
&\le C\Phi(|X_{t-}|^2).
\end{aligned}
\end{equation*}
With a similar estimation as in the proof of Theorem \ref{thm.nonexplosion}, we have
\[I_1 \le CE[\int_0^{t\wedge\tilde \tau_R\wedge\tau_M}\Phi(|X_s|^2)ds].\]
Now we estimate $I_2$. By Assumption \ref{assumption.nocomeback}, it is easy to see that
\[|2\left\langle X_{t-},h\right\rangle+|h|^2| \le C(|X_{t-}|^2\rho(|X_{t-}|^2)+1)\]
and
\[\frac{1}{C}\le\frac{|X_{t-}|^2+1}{|X_{t-}+h(X_{t-},u)|^2+1}\le C\]
Let
\[(*):=\Phi(|X_{t-}+h(X_{t-},u)|^2)-\Phi(|X_{t-}|^2)-\Phi'(|X_{t-}|^2)(2\left\langle X_{t-},h\right\rangle+|h|^2).\]
We see that
\[|\Phi'(|X_{t-}|^2)||2\left\langle X_{t-},h\right\rangle+|h|^2| \le  C\Phi(|X_{t-}|^2),\]
and
\[\Phi(|X_{t-}+h|^2)=\Phi(|X_{t-}|^2)\exp(\int_{|X_{t-}+h|^2}^{|X_{t-}|^2}\frac{1}{s\rho(s)+1}ds).\]
If $|X_{t-}+h| \ge |X_{t-}|$, then $\Phi(|X_{t-}+h|^2)\le\Phi(|X_{t-}|^2)$. In the other case, we have
\[\exp(\int_{|X_{t-}+h|^2}^{|X_{t-}|^2}\frac{1}{s\rho(s)+1}ds) \le \exp(\int_{|X_{t-}+h|^2}^{|X_{t-}|^2}\frac{1}{s+1}ds)=\frac{|X_{t-}|^2+1}{|X_{t-}+h|^2+1}\le C.\]
Hence we get
\[(*) \le C\Phi(|X_{t-}|^2).\]
It implies that
\[E[\Phi(|X_{t\wedge\tilde \tau_R\wedge\tau_M}|^2)]=\Phi(|x_0|^2)+CE[\int_0^{t\wedge\tilde \tau_R\wedge\tau_M}\Phi(|X_s|^2)ds]. \]
By Gronwall lemma, we get
\[E[\Phi(|X_{t\wedge\tilde \tau_R\wedge\tau_M}|^2)]\le\Phi(|x_0|^2)e^{CT}.\]
Letting $M \to \infty$, by Fatou lemma, one has
\[E[\Phi(|X_{t\wedge\tilde \tau_R}|^2)]\le\Phi(|x_0|^2)e^{CT}.\]
Since $\Phi$ is decreasing, this gives that
\[\mathcal P(\tilde \tau_R \le t)\Phi(R^2) \le E[\Phi(|X_{t\wedge\tilde \tau_R}|^2)]\le\Phi(|x_0|^2)e^{CT}\]
Therefore
\[\mathcal P(\inf_{0\le s < t}|X_s(x_0)|\le R)\le e^{Ct}\exp(-\int_{R^2}^{|x_0|^2}\frac{1}{s\rho(s)+1}ds)\]
which tends to zero as $|x_0| \to \infty$. Thus we prove that (\ref{nocomeback}) holds.
\end{proof}
\section{Criteria for pathwise uniqueness and comparison principle}
Inspired from the paper \cite{Bahlali_Hakassou_Ouknine_2015} of Bahlali et al, we consider the pathwise uniqueness under the condition below.
\begin{assumption}\label{assumption.localLip}
There exists $C>0$, $\sigma>0$ and $K \ge e$ such that, for any integer $N>K$,
\begin{equation*}
\begin{aligned}
&|f(x)-f(y)|\le C\log N|x-y|+C\frac{\log N}{N^{\sigma}},\\
&\left \|g(x)-g(y)\right \|^2 \le C\log N|x-y|^2+C\frac{\log N}{N^{2\sigma}},\\
\end{aligned}
\end{equation*}
for all $x,y$ that $|x|,|y|\le N$.\\
For $h$, there exists $p\ge1$ and a constant $C(p)$ depending on $p$ such that
\begin{equation}\label{assumption.h.p-power}
\int_{\mathcal U}|h(x,u)-h(y,u)|^{2p}\mu(du) \le C(p)\log N|x-y|^{2p}+C(p)\frac{\log N}{N^{2\sigma p}},\\
\end{equation}
for all $x,y$ that $|x|,|y|\le N$.
\end{assumption}
Typical example for Assumption \ref{assumption.localLip} is
\[(f,g,h)=(x\log |x|,x\sqrt{|\log |x||},x\sqrt{|\log |x||}u),\]
in the one-dimensional case for $p=2$ if one has that $\int_{\mathcal U}|u|^2\mu(du) <+\infty$. We only check it for $f(x)=x\log|x|$ since it is similar to check for the other two and can be found in \cite{Bahlali_Hakassou_Ouknine_2015}. Thanks to triangular inequalities, it suffices to treat separately two cases: $0 \le |x|,|y|\le\frac{1}{N}$ and $\frac{1}{N}\le|x|,|y|\le N$. In the first case, we see that $|f(x)-f(y)|\le|f(x)|+|f(y)|\le2\frac{\log N}{N}$. While in the other case, one gets that $|f(x)-f(y)| \le (1+\log N)|x-y|$. Combining these two, we see that $f(x)=x \log|x|$ satisfies the assumption. In the next section, we need to assume that (\ref{assumption.h.p-power}) holds for all $p\ge1$. If we know that $\int_{\mathcal U}|u|^{2p}\mu(du) <+\infty$ for any $p \ge1$, then one example that satisfying the assumption is $h(x,u)=\tilde h(x)u$ with $\tilde h$ defined as
\[\tilde h(x)=x(1-exp(-(x-1)^2))\log|\log|x||.\]
One can check that $|\tilde h(x)-\tilde h(y)| \le C \log \log N|x-y|+C\frac{\log N}{N}$ for $N$ sufficiently large. Thus one has
\begin{equation*}
\begin{aligned}
|\tilde h(x)-\tilde h(y)|^p\le& C(\log \log N)^p|x-y|^p+C\frac{(\log N)^p}{N^p}\\
\le&C(p)\log N|x-y|^p+C(p)\frac{\log N}{N^{\frac{p}{2}}}.
\end{aligned}
\end{equation*}
The following lemma is needed for the proof of other theorems in this paper. It shows that, in some sense, the assumption is stable under localization.
\begin{lemma}\label{lemma.cutoff}
For $R>0$, consider a smooth function $\phi_R:\mathbb R^m \to \mathbb R$ satisfying
\[0\le \phi_R\le1, \phi_R(x)=1 \text{ for } |x|\le R+1, \phi_R(x)=0 \text{ for } |x|\ge R+3,\]
and $|\phi'_R(x)|\le1$. If the coefficients $(f,g,h)$ satisfy the Assumption \ref{assumption.localLip}, then the coefficients $(\phi_Rf,\phi_Rg,\phi_Rh)$ also satisfy  the Assumption \ref{assumption.localLip}.
\end{lemma}
\begin{proof}
We only check that $\phi_Rf$ satisfies the assumption, since the proof for the other two is similar. Set $M:=\sup_{R+1\le|x|\le R+3}|b(x)|$. If $0\le |x|,|y|\le R+1$ or $|x|,|y|\ge R+3$, it is obvious that
\[|\phi_R(x)f(x)-\phi_R(y)f(y)|\le C\log N|x-y|+C\frac{\log N}{N^{\sigma}}.\]
If $R+1\le |x|,|y|\le R+3$, then
\begin{equation*}
\begin{aligned}
|\phi_R(x)f(x)-\phi_R(y)f(y)|\le &|\phi_R(x)f(x)-\phi_R(x)f(y)|+|\phi_R(x)f(y)-\phi_R(y)f(y)|\\
\le&|f(x)-f(y)|+M|\phi_R(x)-\phi_R(y)|\\
\le&C\log N |x-y|+C\frac{\log N}{N^{\sigma}}+M|x-y|\\
\le&(C+M)\log N |x-y|+C\frac{\log N}{N^{\sigma}}.
\end{aligned}
\end{equation*}
In the general case, for example, if $0\le |x|\le R+1$ and $|y|\ge R+3$, then one can find $z_1$ and $z_2$ such that $|z_1|=R+1$, $|z_2|=R+3$ and
\[|x-y|=|x-z_1|+|z_1-z_2|+|z_2-y|.\]
Then
\begin{equation*}
\begin{aligned}
|f(x)-f(y)| \le& |f(x)-f(z_1)|+|f(z_1)-f(z_2)|\\
\le&C\log N |x-z_1|+(C+M)\log N|z_1-z_2|+2C\frac{\log N}{N^{\sigma}}\\
\le&(C+M)\log N|x-y|+2C\frac{\log N}{N^{\sigma}}.
\end{aligned}
\end{equation*}
Thus we see that $\phi_Rf$ satisfies the assumption.
\end{proof}
Now we prove that the pathwise uniqueness holds for SDE (\ref{SDE}).
\begin{theorem}\label{thm.localLip}
If Assumption \ref{assumption.localLip} holds with $p=2$, then pathwise uniqueness  holds for the SDE (\ref{SDE}).
\end{theorem}
\begin{proof}
Let $X_t$ and $Y_t$ be two solutions of (\ref{SDE}) with the same initial value. For $N>0$, we define the  stopping time
\[\tau_N:=\inf\{t>0,|X_t| \text{ or }|Y_t| \ge N\}.\]
Since the solutions are assumed to be conservative, then $\tau_N$ tends to $+\infty$ as $N \to +\infty$.\\
Applying It\^o-formula to $|X_t-Y_t|^2$, we get
\begin{equation*}
\begin{aligned}
d|X_t-Y_t|^2=&2\left \langle X_{t-}-Y_{t-},f(X_{t-})-f(Y_{t-})\right \rangle dt+\left \|g (X_{t-})-g(Y_{t-})\right\|^2dt\\
&+\int_{\mathcal U}|h(X_{t-},u)-h(Y_{t-},u)|^2\mu(du)dt\\&
+2\left \langle X_{t-}-Y_{t-},g(X_{t-})-g(Y_{t-})\right \rangle dW_t\\
&+\int_{\mathcal U}|X_{t-}-Y_{t-}+(h(X_{t-},u)-h(Y_{t-},u))|^2-|X_{t-}-Y_{t-}|^2\tilde N(du,dt).
\end{aligned}
\end{equation*}
Taking expectation, we get that
\begin{equation*}
\begin{aligned}
E[|X_{t \wedge \tau_N}-Y_{t \wedge \tau_N}|^2]\le &E[\int_0^t 2|f(X_{(s\wedge \tau_N)-})-f(Y_{(s\wedge \tau_N)-})||X_{(s\wedge \tau_N)-}-Y_{(s\wedge \tau_N)-}|ds]\\
&+E[\int_0^t \left \|g(X_{(s\wedge \tau_N)-})-g(Y_{(s\wedge \tau_N)-}) \right \|^2ds]\\
&+E[\int_0^t \int_{\mathcal U}|h(X_{(s\wedge \tau_N)-},u)-h(Y_{(s\wedge \tau_N)-},u)|^2\mu(du)ds].
\end{aligned}
\end{equation*}
By H\"older inequality and Assumption \ref{assumption.localLip}, we have
\begin{equation*}
\begin{aligned}
E[|X_{t \wedge \tau_N}-Y_{t \wedge \tau_N}|^2]\le (1+C\log N)E[\int_0^t |X_{s \wedge \tau_N}-Y_{s \wedge \tau_N}|^2ds]+CT\frac{\log N}{N^{2\sigma}}.
\end{aligned}
\end{equation*}
Thanks to Gronwall lemma, we get
\[E[|X_{t \wedge \tau_N}-Y_{t \wedge \tau_N}|^2]\le CT\frac{\log N}{N^{2\sigma-CT}}.\]
If $T \le \frac{\sigma}{C}$, then, letting $N \to +\infty$, it holds that
\[E[|X_{t}-Y_{t}|^2]\le 0,\]
which implies that, given $t>0$, $X_t=Y_t$ almost surely. Thus almost surely, we have $X_t=Y_t$, for all $t \in Q\cap [0,T]$. Since the path is c\`adl\`ag, we have that $X$ and $Y$ are indistinguishable before $\frac{\sigma}{C}$.\\
Starting again from $\frac{\sigma}{C}$ and applying the same argument as above, we get that almost surely, $X_t=Y_t$ for $t \in [\frac{\sigma}{C},\frac{2\sigma}{C}]$. Repeating this procedure, we finally get that for every $T >0$, the two solutions are indistinguishable.
\end{proof}
Next we consider the comparison principle for (\ref{SDE}). As the example given in previous, we see that it will not hold in general because of the appearance of the random jump. But  some additional assumptions, especially the non-decreasing of $h$, will imply the comparison principle in the one-dimensional case.
\begin{assumption}\label{assumption.comparison}
Let $C>0$ and $\sigma>0$. For any integer $M > e$, one can find a non-decreasing continuous function $r_M$ such that
\begin{equation*}
\begin{aligned}
&|f(x)-f(y)| \le C\log M|x-y|+C\frac{\log M}{M^{\sigma}},\\
&(g(x)-g(y))^2 \le r_M^2(|x-y|),\\
&\int_{\mathcal U}(h(x,u)-h(y,u))^2\mu(du) \le r_M^2(|x-y|),
\end{aligned}
\end{equation*}
for all $|x|,|y| \le M$.The function $r_M$ should also  satisfy $\int_0^1 \frac{du}{r^2_M(u)}=+\infty$.\\
 Moreover the function $h$ is nondecreasing, i.e.
\[h(x,u) \le h(y,u), \text{ for all }x\le y,u\in \mathcal U.\]
\end{assumption}
\begin{theorem}\label{thm.comparison}
In the one-dimensional case, let $(f_1,g,h)$ and $(f_2,g,h)$ be two sets of coefficients that satisfy
Assumption \ref{assumption.comparison}. Furthermore, we have
 \[f^1(x) \le f^2(x), \text{ for all } x \in \mathbb R.\]
 Let $X^1$ and $X^2$ be the associated solutions for (\ref{SDE}). If $X^1_0 \le X^2_0$, then almost surely
\[X^1_t \le X^2_t, \text{ for all } t\in[0,T].\]
\end{theorem}
\begin{proof}
For $n \in \mathbb N$, let $\{a_n\}$ be the sequence defined by: $a_0=1>a_1>a_2>...>a_n>... \to 0$ and satisfies
\[\int_{a_n}^{a_{n-1}}\frac{du}{r^2_M(u)}=n.\]
For each $n$, let $\phi_n$ be  a non-negative continuous function such that its support is contained in $(a_{n},a_{n-1})$ and satisfies
\[\int_{a_{n}}^{a_{n-1}}\phi_n(u)du=1 \text{ and }r^2_M(u)\phi_n(u) \le \frac{2}{n}.\]
For every $n$, the function $\psi_n(x):=\{\int_0^{x}\int_0^{y}\phi_n(z)dzdy\}1_{(0,+\infty)}(x)$ has the following properties,
\[\psi_n \in C^2(\mathbb R), \psi_n(x) \uparrow x^+ \text{ when } n \to +\infty, \psi'_n(x)\le1 \text{ and } r_M^2(|x|)\psi_n''(x)\le \frac{2}{n}.\]
Define the stopping time $\tau_M$:
\[\tau_M:=\inf\{t>0,|X^1_t| \text{ or }|X^2_t| \ge M\}.\]
Using It\^o-formula and taking expectation, it follows that
\begin{equation*}
\begin{aligned}
E[\psi_n(X^1_{t\wedge \tau_M}-X^2_{t\wedge \tau_M})]=I_1+I_2+I_3,
\end{aligned}
\end{equation*}
where
\begin{equation*}
\begin{aligned}
&I_1=E[\int_0^t\psi'_n(X^1_{s\wedge \tau_M}-X^2_{s\wedge \tau_M})(f_1(X^1_{s\wedge \tau_M})-f_2(X^2_{s\wedge \tau_M}))ds],\\
&I_2=\frac{1}{2}E[\int_0^t\psi''_n(X^1_{s\wedge \tau_M}-X^2_{s\wedge \tau_M})(g(X^1_{s\wedge \tau_M})-g(X^2_{s\wedge \tau_M}))^2ds],\\
&I_3=E[\int_0^{t}\int_{\mathcal U}D_{X^1_{(s\wedge \tau_M)-}-X^2_{(s\wedge \tau_M)-}+h_0((s\wedge \tau_M)-,u)}\psi_n(X^1_{(s\wedge \tau_M)-}-X^2_{(s\wedge \tau_M)-})\mu(du)ds],\\
&h_0(s,u)=h(X^1_s,u)-h(X^2_s,u).
\end{aligned}
\end{equation*}
By Assumption \ref{assumption.comparison} and the construction of $\psi_n$, we see that
\begin{equation}\label{estimation.I_1}
\begin{aligned}
I_1&\le E[\int_0^t\psi'_n(X^1_{s\wedge \tau_M}-X^2_{s\wedge \tau_M})(f_1(X^1_{s\wedge \tau_M})-f_1(X^2_{s\wedge \tau_M}))ds]\\
&\le C\log M E[\int_0^t(X^1_{s\wedge \tau_M}-X^2_{s\wedge \tau_M})^+ds]+C\frac{\log MT}{M^{\sigma}}.
\end{aligned}
\end{equation}
and
\begin{equation}\label{estimation.I_2}
\begin{aligned}
I_2 \le \frac{C}{2}E[\int_0^t\psi''_n(X^1_{s\wedge \tau_M}-X^2_{s\wedge \tau_M})r_M^2(|X^1_{s\wedge \tau_M}-X^2_{s\wedge \tau_M}|)ds]\le \frac{CT}{2n}.
\end{aligned}
\end{equation}
Note that, by Taylor's expansion,
\begin{equation*}
\begin{aligned}
D_y\psi_n(x) &=(y-x)^2\int_0^1 \psi_n''(x+t(y-x))(1-t)dt\\
 &\le \frac{2(y-x)^2}{n}\int_0^1 r^{-2}_M(|x+t(y-x)|)(1-t)dt\\
&\le \frac{(y-x)^2}{n}r^{-2}_M(|x|),
\end{aligned}
\end{equation*}
if $x(y-x) \ge 0$.\\
Thus, by the monotonicity of $h$, we have
\begin{equation}\label{estimation.I_3}
I_3\le E[\int_0^t\frac{1}{n}r^{-2}_M(|X^1_{s\wedge \tau_M}-X^2_{s\wedge \tau_M}|)\int_{\mathcal U} h^2_0(s\wedge \tau_M,u)\mu(du)ds]\le \frac{CT}{n}.
\end{equation}
Combining (\ref{estimation.I_1}), (\ref{estimation.I_2}) and (\ref{estimation.I_3}), we get
\[E[\psi_n(X^1_{t\wedge \tau_M}-X^2_{t\wedge \tau_M})]\le C\log M E[\int_0^t(X^1_{s\wedge \tau_M}-X^2_{s\wedge \tau_M})^+ds]+C\frac{\log MT}{M^{\sigma}}+\frac{CT}{n}.\]
Letting $n \to +\infty$, we have
\[E[(X^1_{t\wedge \tau_M}-X^2_{t\wedge \tau_M})^+]\le C\log M E[\int_0^t(X^1_{s\wedge \tau_M}-X^2_{s\wedge \tau_M})^+ds]+C\frac{\log MT}{M^{\sigma}}.\]
Gronwall lemma implies that
\[E[(X^1_{t\wedge \tau_M}-X^2_{t\wedge \tau_M})^+]\le CT\frac{\log M}{M^{\sigma-CT}}.\]
If $T \le \frac{\sigma}{2C}$, letting $M \to +\infty$, we have
\[E[(X^1_{t}-X^2_{t})^+]\le 0.\]
Thus we prove that almost surely $X^1_t \le X^2_t$ for all $t \in [0, \frac{\sigma}{2C}]$. Starting from $\frac{\sigma}{2C}$ and applying the same argument as above, we get that $X^1_t \le X^2 _t$ for all $t \in [\frac{\sigma}{2C},\frac{\sigma}{C}]$. Repeating this procedure, one can finally show that for every $T >0$, almost surely $X^1_t \le X^2_t$ for all $t \in [0,T]$.
\end{proof}
$\qquad$\\
As an application of Theorem \ref{thm.nonexplosion} and \ref{thm.localLip}, we will study one particular case:
\begin{equation}\label{SDE.log}
\begin{aligned}
X_t=x&+\int_0^{t}X_{s-}\log|X_{s-}|ds +\int_0^t X_{s-}\sqrt{|\log|X_{s-}||}dW_s\\
&+\int_0^{t}\int_{\mathcal U}X_{s-}\sqrt{|\log|X_{s-}||}u\tilde N(du,ds).
\end{aligned}
\end{equation}
Here $x \in \mathbb R$, $W_t$ is one dimensional standard Brownian motion and $N(du,dt)$ is a Poisson random measure with the intensity  measures satisfies \\
$\int_{\mathcal U} |u|^2\mu(du) < +\infty$.\\
\qquad\\
Note that the coefficients satisfy Assumption \ref{assumption.nonexplosion} with $\rho(s)=\log(s)$. And, as point out in \cite{Bahlali_Hakassou_Ouknine_2015} (Proposition 4.1), Assumption \ref{assumption.localLip} is also satisfied.
\begin{theorem}\label{thm.strong.solution}
Let $T>0$ be fixed. For any $x \in \mathbb R$, the SDE (\ref{SDE.log}) admits a unique strong solution.
\end{theorem}
\begin{proof}
For any $N >0$, one can find a smooth cut-off function $\phi_N$ satisfying the assumption in Lemma \ref{lemma.cutoff}. Consider the following SDE:
\begin{equation}\label{SDE.log.N}
\begin{aligned}
dX_t=&\phi_N(X_{t-})X_{t-}\log|X_{t-}|dt+\phi_N(X_{t-})X_{t-}\sqrt{|\log|X_{t-}||}dW_t\\
&+\int_{\mathcal U}\phi_N(X_{t-})X_{t-}\sqrt{|\log|X_{t-}||}u\tilde N(du,dt).
\end{aligned}
\end{equation}
Note that the coefficients of (\ref{SDE.log.N}) are continuous and bounded. By Lemma \ref{lemma.cutoff}, they also satisfy Assumption \ref{assumption.localLip}.  Hence, by Theorem 175 in \cite{Situ_2005}, we see that (\ref{SDE.log.N}) has a weak solution. Theorem \ref{thm.localLip} also implies that the pathwise uniqueness holds for this SDE. It is well-known that the existence of weak solutions and pathwise uniqueness imply the existence of strong solutions (See  \cite[Theorem 137]{Situ_2005} and Barczy et al \cite{Barczy_Li_Pap_2015}). Hence for any $N >0$, we have a strong
solution $X^N$ for (\ref{SDE.log.N}). Define the stopping time $\tau_N:=\inf\{t>0,|X^N_t| \ge N\}$. Again, by the pathwise uniqueness, we see that
\[X^{N_1}_t=X^{N_2}_t, \text{ if } t < \tau_{N_1} \wedge \tau_{N_2}.\]
Thus we define the process $X$:
\[X_t=X^N_t, \text{ if } t< \tau_N \text { for some $N$}.\]
Then we see that $X_t$ is a solution for (\ref{SDE.log}) up to a lifetime $\zeta:=\lim_{N \to \infty}\tau_N$. But Theorem \ref{assumption.nonexplosion} implies that $\zeta=+\infty$ almost surely. Thus $X$ is a strong solution for the SDE (\ref{SDE.log}).
\end{proof}
\section{Non-contact property and continuity of the stochastic flow}
In this section, we consider the stochastic flow associated with (\ref{SDE}). We will prove that the solutions $X_t(x)$ with initial value $x$ satisfy the non-contact property and have a modification continuous with respect to $x$.  The following assumption is needed.
\begin{assumption}\label{assumption.nocontact}
The map $\Gamma_u:x \to x+h(x,u)$ is homeomorphic. For the inverses $\{\Lambda_u \}$, there exists a positive constant $K>0$
such that
\[|\Lambda_u(x)|\le K(1+|x|) \text{ and } |\Lambda_u(x)-\Lambda_u(y) | \le K|x-y|, \text{ for all } x,y \in \mathbb R^m,u \in \mathcal U.\]
\end{assumption}
For the proof of Theorem \ref{thm.nocontact}, we need the following lemma.
\begin{lemma}
Set $F(x)=(\varepsilon+|x|^2)^{-1}$. If Assumption \ref{assumption.nocontact} is satisfied and Assumption \ref{assumption.localLip} holds for any $p \ge 1$, then we have: \\
(1) $F(x-x'+h(x,u)-h(x',u))\le (1+K^2)F(x-x')$ for all $x,x' \in \mathbb R^m,u \in \mathcal U$,\\
(2) There exists a constant $c'$ independent of $\varepsilon$ such that, for all $|x|,|y|\le N$
\begin{equation}\label{lemma.integrate.F}
\scriptsize
\begin{aligned}
\int_{\mathcal U}&F(x-x'+h(x,u)-h(x',u))-F(x-x')-\left \langle \nabla F(x-x'),h(x,u)-h(x',u) \right \rangle\mu(du)\\
\le &c'\{\log N F(x-x')+\log N\frac{F^2(x-x')}{N^{2\sigma}}+\log N\frac{F^{\frac{5}{2}}(x-x')}{N^{3\sigma}}\}
\end{aligned}
\end{equation}
\end{lemma}
\begin{proof}
We shall prove the first assertion. Since $\Lambda_u$ is uniformly Lipschitz continuous, we have
\[|\Lambda_u(y)-\Lambda_u(y')| \le K|y-y'|.\]
Substituting $y=\Gamma_u(x)$ and $y'=\Gamma_u(x')$ in the above inequality, we obtain that
\[|x-x'|\le K|\Gamma_u(x)-\Gamma_u(x')|, \text{ for any }x,x'.\]
Then it holds that $\varepsilon+|x-x'|^2 \le (1+K^2)(\varepsilon+|\Gamma_u(x)-\Gamma_u(x')|^2)$ for any $\varepsilon>0$, which implies that
\begin{equation*}
\begin{aligned}
F(x-x'+h(x,u)-h(x',u))=&(\varepsilon+|\Gamma_u(x)-\Gamma_u(x')|^2)^{-1}\\
\le& (1+K^2)(\varepsilon+|x-x'|^2)^{-1}\\
=& (1+K^2)F(x-x').
\end{aligned}
\end{equation*}
Now we prove the second assertion. Set $w=x-x'$ and $k=g(x,u)-g(x',u)$. Then we have the following equality
\begin{equation*}
\begin{aligned}
&F(w+k)-F(w)-\left \langle \nabla F(w),k\right \rangle\\
=&-|k|^2F(w+k)F(w)-2\left \langle  w,k\right \rangle F(w)(F(w+k)-F(w)).
\end{aligned}
\end{equation*}
It holds that $F(w+k)-F(w)=-(|k|^2+2\left \langle w,k\right\rangle)F(w+k)F(w)$. \\
Since $|w| \le F^{-\frac{1}{2}}(w)$, we have
\begin{equation*}
\begin{aligned}
|F(w+k)-F(w)|\le& |k|^2F(w+k)F(w)+2|w||k|F(w+k)F(w)\\
\le&|k|^2(1+K^2)F^2(w)+2|k|(1+K^2)F^{\frac{3}{2}}(w),
\end{aligned}
\end{equation*}
where we use the first assertion to get the second inequality.\\
Therefore,
\begin{equation*}
\begin{aligned}
&|F(w+k)-F(w)-\left \langle \nabla F(w),k\right \rangle|\\
\le&|k|^2F(w+k)F(w)+2|w||k|F(w)|F(w+k)-F(w)|\\
\le&5|k|^2(1+K^2)F^2(w)+2|k|^3(1+K^2)F^{\frac{5}{2}}(w).
\end{aligned}
\end{equation*}
Now integrate both sides of the above inequality with respect to the measure $\mu$. According to Assumption  \ref{assumption.localLip}, we see that
\begin{equation*}
\scriptsize
\begin{aligned}
\int_{\mathcal U}F(x-x'+h(x,u)-h(x',u))-F(x-x')-\left \langle \nabla F(x-x'),h(x,u)-h(x',u) \right \rangle\mu(du)\\
\le 7(1+K^2)C\log N F(w)+5C(1+K^2)\log N\frac{F^2(w)}{N^{2\sigma}}+2C(1+K^2)\log N\frac{F^{\frac{5}{2}}(w)}{N^{3\sigma}}
\end{aligned}
\end{equation*}
Thus we obtain (\ref{lemma.integrate.F}).
\end{proof}
Now we prove the non-contact property.
\begin{theorem}\label{thm.nocontact}
Assume that  Assumptions \ref{assumption.nocontact} is satisfied and Assumption \ref{assumption.localLip} holds for any $p \ge 1$. For $x \neq y$, let $X_t(x)$ and $X_t(y)$ be the solutions of SDE (\ref{SDE}) starting, respectively from $x$ and $y$. We assume that that the strong solutions for (\ref{SDE}) are conservative (i.e. $\mathcal P(\tau=+\infty)=1$). Then  we have  almost surely $X_t(x)\neq X_t(y)$ for all $0\le t\le T$.
\end{theorem}
\begin{proof}
Denote by $X_t(x)$ and $X_t(y)$ the solutions of (\ref{SDE}) starting, respectively, from $x$ and $y$.  We set
\[\tau_N:=\inf\{t >0;|X_t(x)|\text{ or } |Y_t(y)| \ge N\},\]
and
\[\zeta_N\:=\inf\{t>0;|X_t(x)-X_t(y)| \le \frac{1}{N^{\sigma}}\}.\]
Then, as $N$ goes to $+\infty$, we have $\tau_N$ tends to $+\infty$ and $\zeta_N$ tends to $\zeta:=\inf\{t>0;|X_{t-}(x)-X_{t-}(y)|\text{ or } |X_{t}(x)-X_{t}(y)|= 0\}$. Denote by $\varsigma_N:=\zeta_N \wedge \tau_N$.\\
Consider the function $F(x)=(\varepsilon+|x|^2)^{-1}$ for all $\varepsilon>0$.
Set $\eta_t:=X_t(x)-X_t(y)$ and $h_0(s,u):=h(X_s(x),u)-h(X_s(y),u)$. Applying It\^o-formula to $F(\eta_t)$, we get
\[F(\eta_{t \wedge \varsigma_N})=F(\eta_0)+I_1(t\wedge \varsigma_N)+I_2(t\wedge \varsigma_N)+M(t\wedge \varsigma_N)\]
with
\begin{equation*}
\begin{aligned}
I_1(t)=&\int_0^t\left \langle DF(\eta_{s-}),f(X_{s-}(x))-f(X_{s-}(y)))\right \rangle ds\\
&+\frac{1}{2}\int_0^t \left \langle D^2F(\eta_{s-})(g(X_{s-}(x))-g(X_{s-}(y))),g(X_{s-}(x))-g(X_{s-}(y))\right \rangle ds,\\
I_2(t)=&\int_0^t\int_{\mathcal U}D_{\eta_{s-}+h(X_{s-}(x),u)-h(X_{s-}(y),u)}F(\eta_{s-})\mu(du)ds,\\
M(t)=&\int_0^t\left \langle DF(\eta_{s-},g(X_{s-}(x))-g(X_{s-}(y)))\right \rangle dW_s\\
&+\int_0^t\int_{\mathcal U}F(\eta_{s-}+h(X_{s-}(x),u)-h(X_{s-}(y),u))-F(\eta_{s-})\tilde N(du,ds).
\end{aligned}
\end{equation*}
According to Assumption \ref{assumption.localLip}, we see that
\begin{equation}\label{estimation.I_1.lemma}
\begin{aligned}
I_1(t)\le& 2\int_0^t{|\eta_{s-}|F^{2}(\eta_{s-})|f(X_{s-}(x))-f(X_{s-}(y)))|}ds\\
&+\int_0^t[F^{2}(\eta_{s-})+4|\eta_{s-}|^2F^{3}(\eta_{s-})]\left \| g(X_{s-}(x))-g(X_{s-}(y))\right\|^2ds\\
&\le C\log N \int_0^t F(\eta_{s-})ds+C\frac{\log N}{N^{2\sigma}}\int_0^tF^{2}(\eta_{s-})ds.
\end{aligned}
\end{equation}
By (\ref{lemma.integrate.F}), we see that
\begin{equation}\label{estimation.I_2.lemma}
\begin{aligned}
I_2(t) \le& C\int_0^t\{\log N F(\eta_{s-})+\log N\frac{F^2(\eta_{s-})}{N^{2\sigma}}+\log N\frac{F^{\frac{5}{2}}(\eta_{s-})}{N^{3\sigma}}\}ds.
\end{aligned}
\end{equation}
Since $\frac{1}{N^{2\sigma}} \le F^{-1}(\eta_{s\wedge \varsigma_N})$ for all $s <\varsigma_N$, combining (\ref{estimation.I_1.lemma}) and (\ref{estimation.I_2.lemma}), it follows that
\[I_1(t\wedge \varsigma_N)+I_2(t\wedge \varsigma_N)\le C\log N\int_0^{t\wedge \varsigma_N}F(\eta_{s-})ds.\]
Taking expectation, we have
\[E[F(\eta_{t\wedge \varsigma_N})]\le C\log NE[\int_0^tF(\eta_{s\wedge \varsigma_N})ds].\]
Using Gronwall lemma, we obtain that
\[E[F(\eta_{t\wedge \varsigma_N})]\le F(\eta_0)N^{Ct},\]
that is
\[E[(\varepsilon+ |X_{t\wedge \varsigma_N}(x)-X_{t\wedge \varsigma_N}(y)|^2)^{-1}] \le (\varepsilon +|x-y|^2)^{-1}N^{Ct}.\]
Letting $\varepsilon \to 0$ in the previous inequality, we get
\begin{equation}\label{estimation.thm.x^-2}
E[ |X_{t\wedge \varsigma_N}(x)-X_{t\wedge \varsigma_N}(y)|^{-2}]\le |x-y|^{-2}N^{Ct}.
\end{equation}
On the set $\{\zeta_N\le t \wedge \tau_N\}$, since the path is right continuous, we see that $|\eta_{\zeta_N}| \le \frac{1}{N^{\sigma}}$. Combining with (\ref{estimation.thm.x^-2}), we get that
\[N^{2\sigma}P(\zeta_N \le t\wedge\tau_N) \le |x-y|^{-2}N^{Ct},\]
that is
\[\mathcal P(\zeta_N \le t\wedge\tau_N) \le |x-y|^{-2}N^{Ct-2\sigma}.\]
Letting $N \to +\infty$ in the previous inequality, we obtain for $t<\frac{2\sigma}{C}$,
\[\mathcal P[\zeta \le t]=0.\]
Now starting from $\frac{2\sigma}{C}$ and using the same argument as above, we get for any $t\in[\frac{2\sigma}{C},\frac{4\sigma}{C}]$,
\[\mathcal P[\zeta\le t]=0.\]
It is easy to see that $T_k=\frac{2k\sigma}{C}$ goes to $+\infty$ as $k$ tends to $+\infty$. Arguing recursively on $k$, one can prove that, for any $t \ge 0$
\[\mathcal P(\zeta\le t)=0.\]
\end{proof}
Now we consider the continuity of the stochastic flow associated with (\ref{SDE}). The following theorem is essential.
\begin{theorem}\label{thm.continuous.bounded} In addition to the assumption of Theorem \ref{thm.nocontact}, we also assume that the coefficients $f$ and $g$ are uniformly bounded and \[\sup_{x}\int_{\mathcal U}|h(x,u)|^p\mu(du) <+\infty\]
for any $p>2$. Then, for any $R,T>0$ and each $p>2$ there exists a positive constant $C(p,R,T)$ such that for any $|x|,|y|\le R$ and any $t \in[0,T]$,
\[E[|X_t(x)-X_t(y)|^{2p}]\le C(p,R,T)[|x-y|^{\frac{5p}{2}}+|x-y|^{\frac{p}{2}}+|x-y|^{2p}].\]
\end{theorem}
For the proof of Theorem \ref{thm.continuous.bounded}, we need the following two lemmas.
\begin{lemma}\label{lemma.estimation.stoppingtime}
Under the assumptions of Theorem \ref{thm.continuous.bounded}, for any $p>1$, we have \[\mathcal P(\zeta_N \le t \wedge \tau_N) \le (1\vee K)|x-y|^{2p}N^{C(p)t+2p\sigma},\]
for some constant $C(p)$ depending on $p$. The stopping times $\tau_N$ and $\zeta_N$ are those that defined in the proof of Theorem \ref{thm.nocontact}.
\end{lemma}
\begin{proof}
The processes $\eta_t$, $h_0(t,u)$ and the stopping times $\zeta_N$, $\tau_N$ and $\varsigma_N$ are defined as in the proof of Theorem \ref{thm.nocontact}. Set $r(x)=\varepsilon +|x|^2$ and $R(x)=r^p(x)$ for any $\varepsilon>0$ and $p>1$. Direct computation indicates that the gradient and the Hessian matrix of $R$ are
\[DR(x)=2p(\varepsilon+|x|^2)^{p-1}x,\]
and
\[D^2R(x)=2p(\varepsilon+|x|^2)^{p-1}I+4p(p-1)(\varepsilon+|x|^2)^{p-2}x \otimes x,\]
where $I$ is the identity matrix and $x \otimes x$ is the tensor product of $x$, i.e. $\left \langle (x\otimes x)\xi,\xi\right\rangle=(\left \langle x,\xi\right\rangle)^2$, for any $\xi \in \mathbb R^m$.\\
Thus we see that, for some $\theta \in [0,1]$,
\begin{equation}\label{p-power}
\begin{aligned}
D_{x+y}R(x)=&\left \langle D^2R(x+\theta y)y,y\right \rangle\\
=&2p(\varepsilon+|x+\theta y|^2)^{p-1}|y|^2+4p(p-1)(\varepsilon\\
&+|x+\theta y|^2)^{p-2}|\left \langle x+\theta y,y\right \rangle|^2\\
\le& C(p)\{r^{p-1}(x)|y|^2+r^{p-2}(x)|y|^4+r(x)|y|^{2p-2}+|y|^{2p}\},
\end{aligned}
\end{equation}
with the constant $C(p)$ depending on $p$.\\
Applying It\^o formula to $R(\eta_s)$, we have
\[R(\eta_{t \wedge \varsigma_N})=R(\eta_0)+I_1(t \wedge \varsigma_N)+I_2(t \wedge \varsigma_N)+M(t \wedge \varsigma_N)\]
with
\begin{equation*}
\begin{aligned}
I_1(t)=&\int_0^t\left \langle DR(\eta_{s-}),f(X_{s-}(x))-f(X_{s-}(y))\right \rangle ds\\
&\quad +\frac{1}{2}\int_0^t \left \langle D^2R(\eta_{s-})(g(X_{s-}(x))-g(X_{s-}(y))),g(X_{s-}(x))-g(X_{s-}(y)) \right\rangle ds,\\
I_2(t)=&\int_0^t \int_{\mathcal U}D_{\eta_{s-}+h_0(s-,u)}R(\eta_{s-})\mu(du)ds,\\
M(t)=&\int_0^t\left \langle DR(\eta_{s-}),g(X_{s-}(x))-g(X_{s-}(y))\right \rangle dW_s\\
&\quad+\int_0^t\int_{\mathcal U}[R(\eta_{s-}+h_0(s-,u))-R(\eta_{s-})]\tilde N(du,ds).
\end{aligned}
\end{equation*}
The estimation of $I_1(t)$ is almost the same as that in the proof of Theorem \ref{thm.nocontact}. Now we consider $I_2(t)$. By (\ref{p-power}) and Assumption \ref{assumption.localLip}, we see that
\begin{equation*}
\begin{aligned}
I_2(t)\le& C(p)\int_0^t \int_{\mathcal U} [r^{p-1}(\eta_{s-})|h_0(s-,u)|^2+r^{p-2}(\eta_{s-})|h_0(s-,u)|^4\\
&\qquad \qquad+r(\eta_{s-})|h_0(s-,u)|^{2p-2}+|h_0(s-,u)|^{2p}]ds\\
\le& C(p)\log N\{\int_0^t[r^p(\eta_{s-})+\frac{r^{p-1}(\eta_{s-})}{N^{2\sigma}}
+\frac{r^{p-2}(\eta_{s-})}{N^{4\sigma}}\\
&\qquad \qquad +\frac{r(\eta_{s-})}{N^{(2p-2)\sigma}}+\frac{1}{N^{2p\sigma}}]ds\}.
\end{aligned}
\end{equation*}
Since $\frac{1}{N^{2\sigma}} \le r(\eta_s)$ for $s < \varsigma_N$, we get
\[I_2(t)\le C(p)\log N\int_0^tR(\eta_{s-})ds.\]
Thus, taking expectation, it follows that
\[E[R(\eta_{t \wedge \varsigma_N})] \le R(\eta_0)+C(p)\log NE[\int_0^{t \wedge \varsigma_N}R(\eta_s)]ds.\]
Using Gronwall lemma, we obtain
\begin{equation}\label{estimation.lemma.p-power}
E[R(\eta_{t \wedge \varsigma_N})] \le R(\eta_0)N^{C(p)t}.
\end{equation}
Letting $\varepsilon \to 0$, we get
\begin{equation*}
E[|X_{t \wedge \tau_N \wedge \zeta_N}(x)-X_{t \wedge \tau_N \wedge \zeta_N}(x)|^{2p}]\le |x-y|^{2p}N^{C(p)t}.
\end{equation*}
On the subset $\{\zeta_N \le t\wedge \tau_N\}$, we see that $|\eta_s| \ge \frac{1}{N^{\sigma}}$ for all $s <\zeta_N$. Thus $|\eta_{\zeta_N}| \ge \frac{1}{N^{\sigma}}$ if the path is left continuous at $\zeta_N$. On the other hand, if there is a jump at $\zeta_N$, according to Assumption \ref{assumption.nocontact}, we have
\begin{equation*}
\begin{aligned}
|X_{\zeta_N}(x)-X_{\zeta_N}(y)|=&|\Gamma_u(X_{\zeta_N-}(x))-\Gamma_u(X_{\zeta_N-}(y))|\\
&\ge\frac{1}{K}|X_{\zeta_N-}(x)-X_{\zeta_N-}(y)|\\
&\ge\frac{1}{K}\frac{1}{N^{\sigma}}
\end{aligned}
\end{equation*}
In both cases, it holds that
\[|\eta_{\zeta_N}| \ge (1\wedge\frac{1}{K})\frac{1}{N^{\sigma}}.\]
Thus we have the following inequality
\[(1\wedge\frac{1}{K})\frac{1}{N^{2p\sigma}}\mathcal P(\zeta_N \le t \wedge \tau_N) \le |x-y|^{2p}N^{C(p)t}\]
which implies that
\[\mathcal P(\zeta_N \le t \wedge \tau_N) \le (1\vee K)|x-y|^{2p}N^{C(p)t+2p\sigma}.\]
\end{proof}
The following lemma has been proved in \cite{Kunita_2004} (Theorem 2.11, pp.332).
\begin{lemma}\label{lemma.supreme.p-power}
Consider a d-dimensional semimartingale  with the following decomposition:
\[dX_t=x+\int_0^tf(s-)ds+\int_0^tg(s-)dW_s+\int_0^t \int_{\mathcal U}h(s-,u)\tilde N(du,ds).\]
For any $p \ge 2$, there exists a constant $C(p)$ such that
\begin{equation*}
\begin{aligned}
E[\sup_{0\le s\le t}|X_s|^p]&\le C(p)\big\{|x|^p+E[(\int_0^t|f(s)|ds)^p]+E[(\int_0^t|g(s)|^2ds)^{\frac{p}{2}}]\\
&+E[(\int_0^t\int_{\mathcal U}|h(s,u)|^2\mu(du)ds)^{\frac{p}{2}}]+E[\int_0^t\int_{\mathcal U}|h(s,u)|^p\mu(du)ds]\big\}.
\end{aligned}
\end{equation*}
\end{lemma}
\begin{bf}{Proof of Theorem \ref{thm.continuous.bounded}}:\end{bf}
According to Lemma \ref{lemma.supreme.p-power}, we see that, for any $R,T>0$ and $p \ge 2$
\begin{equation}\label{estimation.thm.supreme.}
M_{p,R,T}:=\sup_{|x|\le R}E[\sup_{0\le t\le T}|X_t(x)|^p] < +\infty.
\end{equation}
Let $X_t(x)$ and $X_t(y)$ be the solutions of SDE (\ref{SDE}) starting, respectively from $x$ and $y$. The definition of the stopping times $\zeta_N$ and $\tau_N$ is the same as that in the proof of Theorem \ref{thm.nocontact}. Set $r_N(x)=\varepsilon+\frac{1}{N^{2\sigma}}+|x|^2$ and $R_N(x)=r^p_N(x)$. Since $r_N(x) \ge \frac{1}{N^{2\sigma}}$, by similar arguments as in the proof of Lemma \ref{lemma.estimation.stoppingtime}, we have the following inequality:
\[E[R_N(\eta_{t \wedge \tau_N})]\le R_N(\eta_0)+C(p)\log NE[\int_0^{t \wedge \tau_N}R_N(\eta_{s \wedge \tau_N})ds]\]
with some constant $C(p)$ only depending on $p$.\\
Thanks to Gronwall lemma, it follows that
\[E[R_N(\eta_{t \wedge \tau_N})]\le R_N(\eta_0)N^{C(p)t},\]
which is
\begin{equation}\label{estimation.thm.p-power}
E[(\varepsilon+\frac{1}{N^{2\sigma}}+|\eta_{t \wedge \tau_N}|^2)^{p}]\le (\varepsilon+\frac{1}{N^{2\sigma}}+|x-y|^2)^pN^{C(p)t}.
\end{equation}
For $t>0$, we set $Y_t(x):=\sup_{0 \le s <t}|X_s(x)|$.\\
Arguing as in \cite{Bahlali_Hakassou_Ouknine_2015}, we show that
\[(\varepsilon+|X_t(x)-X_t(y)|^2)^p=\sum_{N=1}^{+\infty}(\varepsilon+|X_t(x)-X_t(y)|^2)^p1_{\{N-1 \le Y_T(x) \vee Y_T(y)<N\}},\]
which implies that
\begin{equation*}
\scriptsize
\begin{aligned}
&(\varepsilon+|X_t(x)-X_t(y)|^2)^p\\
=&\sum_{N=1}^{+\infty}(\varepsilon+|X_{t\wedge \tau_N}(x)-X_{t\wedge \tau_N}(y)|^2)^p1_{\{N-1 \le Y_T(x) \vee Y_T(y)<N\}}\times1_{\{\zeta_N \le T \wedge \tau_N\}}\\
+&\sum_{N=1}^{+\infty}(\varepsilon+|X_{t\wedge \tau_N\wedge \zeta_N}(x)-X_{t\wedge \tau_N\wedge \zeta_N}(y)|^2)^p1_{\{N-1 \le Y_T(x) \vee Y_T(y)<N\}}
\times1_{\{\zeta_N > T \wedge \tau_N\}}
\end{aligned}
\end{equation*}
Using Cauchy-Schwartz's inequality, we get
\begin{equation*}
\begin{aligned}
&E[(\varepsilon+|X_t(x)-X_t(y)|^2)^{p}]\\
\le&\sum_{N=1}^{+\infty}E[(\varepsilon+|X_{t\wedge \tau_N}(x)-X_{t\wedge \tau_N}(y)|^2)^{2p}]^{\frac{1}{2}}\\
&\times (\mathcal P[N-1 \le Y_T(x) \vee Y_T(y)])^{\frac{1}{4}}(\mathcal P[\zeta_N \le T \wedge \tau_N])^{\frac{1}{4}}\\
&+\sum_{N=1}^{+\infty}E[(\varepsilon+|X_{t\wedge \tau_N\wedge \zeta_N}(x)-X_{t\wedge \tau_N\wedge \zeta_N}(y)|^2)^{2p}]^{\frac{1}{2}}\\
&\times (\mathcal P[N-1 \le Y_T(x) \vee Y_T(y)])^{\frac{1}{4}}(\mathcal P[\zeta_N > T \wedge \tau_N])^{\frac{1}{4}}
\end{aligned}
\end{equation*}
(\ref{estimation.lemma.p-power}) and (\ref{estimation.thm.p-power}) indicates that
\[E[(\varepsilon+|X_{t \wedge \tau_N \wedge \zeta_N}(x)-X_{t \wedge \tau_N \wedge \zeta_N}(y)|^{2})^{2p}]\le (\varepsilon+|x-y|^{2})^{2p}N^{C(2p)t}\]
and
\[E[(\varepsilon+\frac{1}{N^{2\sigma}}+|\eta_{t \wedge \tau_N}|^2)^{2p}]\le (\varepsilon+\frac{1}{N^{2\sigma}}+|x-y|^2)^{2p}N^{C(2p)t}.\]
By (\ref{estimation.thm.supreme.}), we see that, for any $q\ge2$,
\[P[N-1 \le Y_T(x) \vee Y_T(y)]\le \frac{M_{R,q,T}}{(N-1)^q}.\]
Lemma \ref{p-power} shows that
\[\mathcal P(\zeta_N \le t \wedge \tau_N) \le (1\vee K)|x-y|^{2p}N^{C(p)t+2p\sigma}.\]
Thus we have
\begin{equation*}
\begin{aligned}
&E[(\varepsilon+|X_t(x)-X_t(y)|^2)^p]\\
\le&\sum_{N=1}^{+\infty}(\varepsilon+\frac{1}{N^{2\sigma}}+|x-y|^2)^pN^{C(2p)t}\\
&\times(\frac{M_{R,q,T}}{(N-1)^q})^{\frac{1}{4}}((1\vee K)|x-y|^{2p}N^{C(p)t+2p\sigma})^{\frac{1}{4}}\\
&+\sum_{N=1}^{+\infty}(\varepsilon+|x-y|^2)^{p}N^{C(2p)t}\times(\frac{M_{R,q,T}}{(N-1)^q})^{\frac{1}{4}}.
\end{aligned}
\end{equation*}
This implies that
\begin{equation}\label{estimation.thm.series}
\begin{aligned}
&E[(\varepsilon+|X_t(x)-X_t(x)|^2)^p]\\
\le&2^{p-1}(\varepsilon+|x-y|^2)^{p}|x-y|^{\frac{p}{2}}(1\vee K)M_{R,q,T}^{\frac{1}{4}}\sum_{N=1}^{+\infty}N^{C(p,T)-\frac{q}{4}}\\
&+2^{p-1}|x-y|^{\frac{p}{2}}(1\vee K)M_{R,q,T}^{\frac{1}{4}}\sum_{N=1}^{+\infty}N^{C(p,T)-2\sigma-\frac{q}{4}}\\
&+(\varepsilon+|x-y|^2)^{p})M_{R,q,T}^{\frac{1}{4}}\sum_{N=1}^{+\infty}N^{C(p,T)-\frac{q}{4}}.
\end{aligned}
\end{equation}
Choosing $q$ sufficiently large, the right hand side of (\ref{estimation.thm.series})  will converge. Thus there exists a positive constant $C(p,R,T)$ such that
\[E[(\varepsilon+|X_t(x)-X_t(y)|^2)^p]\le C(p,R,T)[(\varepsilon+|x-y|^2)^{p}|x-y|^{\frac{p}{2}}+|x-y|^{\frac{p}{2}}+(\varepsilon+|x-y|^2)^{p}].\]
Letting $\varepsilon \to 0$, we get that
\[E[|X_t(x)-X_t(y)|^{2p}]\le C(p,R,T)[|x-y|^{\frac{5p}{2}}+|x-y|^{\frac{p}{2}}+|x-y|^{2p}].\]
\endproof
\begin{remark}\label{remark.skorohod}
Similarly, one can show that
\begin{equation}\label{estimation.remark.supreme}
E[\sup_{0\le s\le t}|X_s(x)-X_s(y)|^{2p}]\le C(p,R,T)[|x-y|^{\frac{5p}{2}}+|x-y|^{\frac{p}{2}}+|x-y|^{2p}]
\end{equation}
Consider $\mathbb D(\mathbb R^m)$ the space of all c\`adl\`ag $\mathbb R^m$-valued fucntion on $\mathbb R^{+}$ equipped with the Skorohod topology induced by the metric $d$ (see Chapter 6 in \cite{Jacob_Shiryaev_1987}). For any two path $x$ and $y$, we have $d(x_{\cdot \wedge t},y_{\cdot \wedge t})\le \sup_{0\le s\le t}|x_s-y_s|$. (\ref{estimation.remark.supreme}) and Kolmogorov theorem imply that there is a version $\bar X_{\cdot \wedge t}(x)$ of $X_{\cdot \wedge t}(x)$ such that almost surely $\bar X_{\cdot \wedge t}(x)$ is continuous in $x$ as a mapping from $\mathbb R^m$ to $\mathbb D(\mathbb R^m)$.
\end{remark}
Now we show the continuity of the stochastic flow.
\begin{theorem}\label{thm.continuous.general} In addition to the assumption of Theorem \ref{thm.nocontact}, assume that $f$ and $g$ are locally bounded and\\
\[\sup_{|x|\le R}\int_{\mathcal U}|h(x,u)|^p\mu(du) <+\infty, \text{ for any $p\ge2$ and $R>0$.} \] \\
Then, for each $t>0$, there is a version $\tilde X_t(x)$ of $X_t(x)$ that is continuous on $\mathbb R^m$ almost surely.
\end{theorem}
\begin{proof}
We will proceed as in \cite{Fang_Zhang_2005}. First, we assume that the coefficients are compactly support in the set $\{|x| \le R\}$. Then, by Theorem \ref{thm.continuous.bounded}
\[E[|X_t(x)-X_t(y)|^{2p}]\le C(p,R,T)[|x-y|^{\frac{5p}{2}}+|x-y|^{\frac{p}{2}}+|x-y|^{2p}].\]
Taking $p>d+1$ and using Kolmogorov theorem, we show that the solution $X_t(x)$ admits a continuous version $\tilde X_t(x)$ in $|x| \le R+1$. Moreover, since the coefficients are compact supported and the pathwise uniqueness holds for (\ref{SDE}), we have $X_t(x)=x$ for all $|x| \ge R+\frac{1}{2}$. Thus we can extend $\tilde X_t(x)$ continuously on $\mathbb R^m$. \\
In the general case, for any $R>0$, we consider a smooth function $\phi_R:\mathbb R^m \to \mathbb R$ satisfying
\[0\le \phi_R\le1, \phi_R(x)=1 \text{ for } |x|\le R+1, \phi_R(x)=0 \text{ for } |x|\ge R+3\]
and $|\phi'_R(x)|\le1$.
Consider the following SDE:
\begin{equation}
\begin{aligned}
X_t=x&+\int_0^t\phi_R(X_{s-})f(X_{s-})ds+\int_0^t\phi_R(X_{s-})g(X_{s-})dW_s\\
&+\int_0^t\int_{\mathcal U}\phi_R(X_{s-})h(X_{s-},u)\tilde N(du,ds).\label{SDE.R}
\end{aligned}
\end{equation}
Let $X^R_t(x)$ be a solution of (\ref{SDE.R}). The discussion above indicates that one can assume that $X^R_{\cdot \wedge t}(x)$ is continuous in $x$ as a $\mathbb D(\mathbb R^m)$-valued function. We know that if $\omega_n $ tends to $\omega$ in the Skorohod topology and $\omega$ is continuous at time $t$, then $\omega_n(t)$ tends to $\omega(t)$ (Proposition 2.4 in Section 6.2 of \cite{Jacob_Shiryaev_1987}, pp. 305). By the quasi-left continuity of $X^R$, we see that almost surely $X^R_{t-}=X^R_t$. Thus almost surely $X^R_t(x)$ is continuous in $x$ as a $\mathbb R^m$-valued function. Define the stopping time
\[\tau^R_N(x):=\inf\{t>0,|X^R_t(x)| \ge N\}  \]
and
\[\tau_N(x):=\inf\{t>0,|X_t(x)| \ge N\}. \]
Since the pathwise uniqueness holds, for $|x| \le R$, we have
\[X^R_t(x)=X_t(x), \text{ for any }N  \le R \text{ and } t <\tau^R_N \]
and
\[\tau_N(x)=\tau^R_N(x), \text{ for all } R \ge N.\]
For $|x|\le R$, we define
\[\tilde X_t(x)=X^{R+1}_t(x) \text{ if } t <\tau^{R+1}_{R+1}(x).\]
Then $\tilde X_t(x)$ is a version of $X_t(x)$. Let us prove that $\tilde X_t(x)$ is continuous in $x$ for almost all $\omega$. Fix $x_0$. Then, one can find some $R>|x_0|$ depending on $\omega$, such that $\tau^{R+1}_{R+1}(x) >t+\varepsilon$ for a small $\varepsilon$. By Remark \ref{remark.skorohod}, we can find a neighborhood $B_{\delta}(x_0)$ such that $\tau^{R+1}_{R+1}(x) >t+\varepsilon$ for all $x \in B_{\delta}(x_0)$. Hence, $\tilde X_t(x,\omega)=X^{R+1}_t(x,\omega)$ for all $x \in B_{\delta}(x_0)$ which implies that $\tilde X_t(x)$ is continuous at $x_0$.
\end{proof}
\bibliographystyle{siam}
\bibliography{myreference}

\begin{thebibliography}{10}

\bibitem{Applebaum_2004}
{\sc D.~Applebaum}, {\em L{\'e}vy processes and stochastic calculus}, Cambridge
  Univ. Press, 2004.

\bibitem{Bahlali_Eddahbi_Essaky_2003}
{\sc K.~Bahlali, M.~Eddahbi, and E.~Essaky}, {\em {BSDE} associated with
  {L\'e}vy processes and application to {PDIE}}, Journal of Appl. Math. and
  Stochastic Analysis,  (2003), pp.~1--17.

\bibitem{Bahlali_Hakassou_Ouknine_2015}
{\sc K.~Bahlali, A.~Hakassou, and Y.~Ouknine}, {\em A class of stochastic
  differential equations with super-linear growth and non-{L}ipschitz
  coefficients}, Stochastics and Stochastics Reports,  (2015), pp.~800--847.

\bibitem{Barczy_Li_Pap_2015}
{\sc M.~Barczy, Z.~H. Li, and G.~Pap}, {\em Yamada-{W}atanabe results for
  stochastic differential equations with jumps}, Int. J. Stoch. Anal.,  (2015).

\bibitem{Bass_2003}
{\sc R.~F. Bass}, {\em Stochastic differential equations driven by symmetric
  stable processes}, S{\'e}minaire de Probabilit{\'e}s, XXXVI,  (2003),
  pp.~302--313.

\bibitem{Bass_2004}
\leavevmode\vrule height 2pt depth -1.6pt width 23pt, {\em Stochastic
  differential equations with jumps}, Probability Surveys,  (2004), pp.~1--19.

\bibitem{Fang_Zhang_2005}
{\sc S.~Z. Fang and T.~S. Zhang}, {\em A study of a class of stochastic
  differential equations with non-{L}ipschitizian coefficients}, Probab. Theory
  Related. Fields,  (2005), pp.~356--390.

\bibitem{Fu_Li_2010}
{\sc Z.~F. Fu and Z.~H. Li}, {\em Stochastic equations of non-negative
  processes with jumps}, Stochastic Process. Appl.,  (2010), pp.~306--330.

\bibitem{Hairer_2006}
{\sc M.~Hairer}, {\em Ergodic properties of Markov processes}, Lecture Notes,
  2006.

\bibitem{Ikeda_Watanabe_1989}
{\sc N.~Ikeda and S.~Watanabe}, {\em Stochastic differential equations and
  diffusion processes, Second edition}, North-Holland, Amsterdam, 1989.

\bibitem{Jacob_Shiryaev_1987}
{\sc J.~Jacob and A.~N. Shiryaev}, {\em Limit theorems for stochastic
  processes}, Springer, New York, 1987.

\bibitem{Karatzas_Shreve_1991}
{\sc I.~Karatzas and S.~E. Shreve}, {\em Brownian motion and stochastic
  calculus, Second edition}, Graduate Texts in Mathematics, Springer, 1991.

\bibitem{Komatsu_1982}
{\sc T.~Komatsu}, {\em On pathwise uniqueness of solutions of one-dimensional
  stochastic differential equations of jump type}, Proc. Japan Acad. Ser. A
  Math. Sci.,  (1982), pp.~353--356.

\bibitem{Kunita_2004}
{\sc H.~Kunita}, {\em Stochastic differential equations based on L{\'e}vy
  processes and stochastic flows of diffeomorphisms}, Real and stochastic
  analysis, 2004.

\bibitem{Li_Mytnik_2011}
{\sc Z.~H. Li and L.~Mytnik}, {\em Strong solutions for stochastic differential
  equations with jumps}, Ann. Inst. Henri Poincare Probab. Stat.,  (2011),
  pp.~1055--1067.

\bibitem{Situ_2005}
{\sc R.~Situ}, {\em Theory of Stochastic Differential Equations with Jumps and
  Applications.}, Springer, Berlin, 2005.

\end{thebibliography}
\end{document}